\documentclass{amsart}
\usepackage[numbers,square]{natbib}
\usepackage{amsfonts,amsthm,amsmath,bbm,amssymb,bm}
\usepackage{mathtools,verbatim,mathrsfs}
\usepackage{enumitem}

\numberwithin{equation}{section}

\theoremstyle{plain}
\newtheorem{theorem}{Theorem}[section]
\newtheorem{lemma}[theorem]{Lemma}

\newtheorem*{maintheorem*}{Main Theorem}
\newtheorem*{maintheoremone*}{Theorem 1}
\newtheorem*{maintheoremtwo*}{Theorem 2}
\newtheorem*{cdensitytheorem}{Artin-Chebotarev Density Theorem}
\newtheorem*{hclassfield}{Hilbert Class Field}
\newtheorem*{greentao}{Green-Tao Theorem}
\newtheorem*{dirichlettheorem}{Dirichlet's Theorem}
\newtheorem*{lemmareworked}{Lemma \ref{lem:ifpdividesthensodoesptwo}}
\newtheorem*{artinsymbol}{Artin Symbol}
\newtheorem*{pnt}{Prime Number Theorem}
\newtheorem*{fermat}{Fermat's Two Squares Theorem}

\newcommand{\acd}{generalized PNT}

\theoremstyle{definition}

\newtheorem{examples}[theorem]{Examples}

\newcommand{\defeq}{\vcentcolon=}

\newcommand{\Q}{\mathbb{Q}}
\newcommand{\Z}{\mathbb{Z}}
\newcommand{\N}{\mathbb{N}}
\renewcommand{\P}{\mathbb{P}}

\newcommand{\eps}{\epsilon}

\newcommand{\NPsi}{\N_\Psi}
\newcommand{\PPsi}{\P_\Psi}
\newcommand{\Pram}{\P_{\text{ram}}}
\newcommand{\Phigh}{\P_{\text{high}}}
\newcommand{\Plow}{\P_{\text{low}}}
\newcommand{\Psplit}{\P_{\text{split}}}
\newcommand{\Ploverq}{\P_{L / \Q}}

\newcommand{\normkq}{\mathcal{N}_{K / \Q}}
\newcommand{\idealnorm}{\mathcal{N}}

\newcommand{\intringk}{\mathcal{O}_K}
\newcommand{\koverq}{K \big / \Q}
\newcommand{\intringl}{\mathcal{O}_L}
\newcommand{\loverq}{L \big / \Q}
\newcommand{\loverk}{L \big / K}
\newcommand{\toverk}{T \big / K}
\newcommand{\artin}[3]{\left( \dfrac{#1 / #2}{#3}\right)}

\makeatletter
\def\moverlay{\mathpalette\mov@rlay}
\def\mov@rlay#1#2{\leavevmode\vtop{%
   \baselineskip\z@skip \lineskiplimit-\maxdimen
   \ialign{\hfil$\m@th#1##$\hfil\cr#2\crcr}}}
\newcommand{\charfusion}[3][\mathord]{
    #1{\ifx#1\mathop\vphantom{#2}\fi
        \mathpalette\mov@rlay{#2\cr#3}
      }
    \ifx#1\mathop\expandafter\displaylimits\fi}
\makeatother

\newcommand{\cupdot}{\charfusion[\mathbin]{\cup}{\cdot}}
\newcommand{\bigcupdot}{\charfusion[\mathop]{\bigcup}{\cdot}}

\newcommand{\gal}{\text{Gal}}
\newcommand{\galloverk}{\gal \left( \loverk \right)}

\newcommand{\gaussint}{\Z[i]}
\newcommand{\gaussfield}{\Q(i)}

\newcommand{\baseelt}{\omega}
\newcommand{\pprime}{\mathfrak{p}}
\newcommand{\bprime}{\mathfrak{B}}

\newcommand{\cclass}[1]{\langle {#1} \rangle}

\newcommand{\den}{\mathbf{d}}
\newcommand{\upden}{\overline{\mathbf{d}}}
\newcommand{\lowden}{\underline{\mathbf{d}}}
\newcommand{\denrel}{\mathbf{d_\P}}
\newcommand{\updenrel}{\overline{\mathbf{d_\P}}}
\newcommand{\lowdenrel}{\underline{\mathbf{d_\P}}}
\newcommand{\lden}{\bm{\ell d}}
\newcommand{\ldenrel}{\bm{\ell d_\P}}
\newcommand{\dden}{\bm{\delta}}
\newcommand{\ddenrel}{\bm{\delta_\P}}

\author[D. Glasscock]{Daniel Glasscock}
\address{UMass Lowell, Math Sciences Dept, 1 University Ave, Lowell MA 01854}
\email{daniel\textunderscore glasscock@uml.edu}

\title{Norm forms represent few integers but relatively many primes}

\begin{document}

\begin{abstract}
Norm forms, examples of which include
\begin{align*}x^2 + y^2, \quad x^2 + x y - 57 y^2, \quad \text{and} \quad x^3 + 2 y^3 + 4 z^3 - 6 x y z,\end{align*}
are integral forms arising from norms on number fields. We prove that the natural density of the set of integers represented by a norm form is zero, while the relative natural density of the set of prime numbers represented by a norm form exists and is positive. These results require tools from class field theory, including the Artin-Chebotarev density theorem and the Hilbert class field. We introduce these tools as we need them in the course of the main arguments. This article is expository in nature and assumes only a first course in algebraic number theory.
\end{abstract}

\maketitle

\section{Introduction}\label{sec:defs} 

Denote by $\Z$, $\N$, $\Q$, and $\P$ the sets of integers, positive integers, rational numbers, and (positive, rational) primes, respectively.

Let $K$ be an algebraic number field of degree $d \geq 2$.  Denote by $\normkq: K \to \Q$ the norm and $\intringk$ the ring of integers of $K$. Let $(\baseelt_1, \ldots, \baseelt_d)$ be an integral basis for $\koverq$, that is, a basis for $\intringk$ as a (free) $\Z$-module. The map
\begin{align*}
\Psi: \Z^d &\longrightarrow \Z \\
(x_1, \ldots, x_d) &\longmapsto \normkq \big( x_1 \baseelt_1 + \ldots + x_d \baseelt_d \big)
\end{align*} 
is the \emph{norm form arising from $K$ with the basis $(\baseelt_1, \ldots, \baseelt_d)$}. Take a minute to check that $\Psi$ is a homogeneous, degree $d$, integral-coefficient polynomial in $d$ variables. The following are some examples of norm forms.\footnote{For proof that the indicated tuples are indeed integral bases of the associated number fields, see  \cite[Chapter 2]{marcusbook}, Corollary 2 to Theorem 1, Exercises 35 and 41, and Corollary 2 to Theorem 12.}
\begin{itemize}[leftmargin=0.7cm]
\item For $a \in \Z$ square-free and $a \equiv 2, 3 \pmod 4$, the basis $(1, \sqrt{a})$ for $\Q(\sqrt{a})$ gives rise to the norm form
\[\Psi(x_1,x_2) = x_1^2 - a x_2^2.\]
\item For $a \in \Z \setminus \{1\}$ square-free and $a \equiv 1 \pmod 4$, the basis $\big(1, (1+\sqrt{a})/2 \big)$ for $\Q(\sqrt{a})$ gives rise to the norm form
\[\Psi(x_1,x_2) = x_1^2 + x_1 x_2 - \frac{a-1}{4} \ x_2^2.\]
\item For $a \in \Z \setminus \{0\}$ square-free and $a \not\equiv \pm 1 \pmod 9$, the basis $\big(1, \sqrt[3]{a}, (\sqrt[3]{a})^2 \big)$ for $\Q(\sqrt[3]{a})$ gives rise to the norm form
\[\Psi(x_1,x_2,x_3) = x_1^3 + a x_2^3 + a^2 x_3^3 - 3a x_1x_2x_3.\]
\item For $\omega = e^{2\pi i / 7} + e^{-2\pi i / 7}$, the basis $(1, \omega, \omega^2)$ for $\Q(\omega)$ gives rise to the norm form 
\[\Psi(x_1,x_2,x_3) = x_1^3-x_2^3+x_3^3 -3 x_1 (x_2^2-x_3^2) +3 x_3 (2 x_1+x_2) (x_1+x_3).\]
\item For $\omega = e^{2\pi i / 5}$, the basis $(1, \omega, \omega^2, \omega^3)$ for $\Q(\omega)$ gives rise to the norm form 
\[\Psi(x_1,x_2,x_3,x_4) = \det \left( \begin{array}{cccc}
x_1 & \phantom{x_2}-x_4 & x_4-x_3 & x_3-x_2 \\
x_2 & x_1-x_4 & \phantom{x_2}-x_3 & x_4-x_2 \\
x_3 & x_2-x_4 & x_1-x_3 & \phantom{x_2}-x_2 \\
x_4 & x_3-x_4 & x_2-x_3 & x_1-x_2 
\end{array} \right).\]
\end{itemize}

Given a norm form $\Psi$, let
\[\NPsi = \left\{ |\Psi(z)| \ \middle| \ z \in \Z^d \setminus \{0\} \right\} \quad \text{ and } \quad \PPsi = \NPsi \cap \P\]
be the set of positive integers and prime numbers represented by (the absolute value of) $\Psi$. We will prove the following theorem regarding these sets.

\begin{maintheorem*}
Let $\Psi$ be a norm form.
\begin{enumerate}[label=(\Roman*)]
\item $\Psi$ represents a zero density of the positive integers:
\[ \den(\NPsi) \defeq \lim_{N \to \infty} \frac{\big| \NPsi \cap \{1, \ldots, N \}\big|}{N} =0.\]
\item $\Psi$ represents a positive relative density of the prime numbers:
\[ \denrel(\PPsi) \defeq \lim_{N \to \infty} \frac{\big| \PPsi \cap \{1, \ldots, N \}\big|}{\big| \mathbb{P} \cap \{1, \ldots, N \}\big|} > 0.\]
\end{enumerate}
\end{maintheorem*}

Throughout this note, \emph{density} and \emph{relative density} will refer to the natural asymptotic density $\den$ and its restriction to the primes $\denrel$ as defined above. It is a non-trivial part of the conclusion of the main theorem that the limits in the expressions of the densities exist.  Part (I) of the Main Theorem is both strengthened and generalized in the work of Odoni \cite{odoni}.

By choosing different integral bases, a single number field $K$ gives rise to many different norm forms. While such forms are described by different expressions, they are equivalent up to an invertible, integral change of variables and hence will represent the same set of positive integers. Therefore, the sets $\NPsi$ and $\PPsi$ could be denoted $\N_K$ and $\P_K$ where $K$ is the number field from which $\Psi$ arises.

The main theorem does not hold for all homogeneous, integral forms. The form $x^2-y^2$, for example, represents all integers not congruent to 2 modulo 4, while $x^3+y^3+z^3$ is conjectured to represent all integers not congruent to $\pm 4$ modulo 9 \cite{connvaserstein}. It is easy to see that neither of these forms is a norm form since neither represents 0 uniquely. Norm forms are very special homogeneous, integral forms, and the proof of the main theorem relies heavily on their algebraic origins.

\section{A first example and some basic lemmas}\label{sec:simpleexample}

In this section, we will prove the main theorem for the simplest norm form
\[\Psi(x,y) = x^2 + y^2,\]
establishing in the process some lemmas and the approach used to handle more complicated ones.  We begin by showing that $\Psi$ represents a zero density of the positive integers.

\begin{lemma}\label{lem:ifpdividesthensodoesptwo}
Let $n \in \NPsi$. For all $p \in \P$ congruent to 3 modulo 4, if $p | n$, then $p^2 | n$.
\end{lemma}

\begin{proof}
Write $n = x^2 + y^2$. Suppose $p \in \P$ is congruent to 3 modulo 4 and $p | n$.  If $p \nmid x$, there exists $y_0 \in \Z$ such that $x^2 + y^2 \equiv x^2(1+y_0^2) \pmod p$. Multiplying by the modulo $p$ inverse of $x$, we see that $-1$ is a quadratic residue modulo $p$. This is in contradiction to Euler's Criterion, which states that $-1$ is a quadratic residue modulo an odd prime $p$ if and only if $p$ is congruent to 1 modulo 4. We conclude that $p | x$, whence $p | y$, whereby $p^2 | n$.
\end{proof}

This lemma describes a constraint on the modulo $p^2$ residue class into which the set $\NPsi$ can fall. The next lemma shows that a set which satisfies many such constraints must have zero asymptotic density. Call a subset $P \subseteq \N$ \emph{divergent} if $\sum_{p \in P} p^{-1} = \infty$.

\begin{lemma}\label{lem:zerodensityfromCRT}
If $P \subseteq \P$ is divergent, then the set
\begin{align}\label{eqn:defofa}A = \big\{ n \in \N \ \big| \ \text{for all } p \in P, \text{ if } p | n, \text{ then } p^2 | n \big\}\end{align}
has zero asymptotic density: $\den(A) = 0$.
\end{lemma}

\begin{proof}
For $p \in P$, elements of the set $A$ occupy $p^2 - (p-1)$ modulo $p^2$ residue classes: $A$ avoids the classes $p, 2p, \ldots, (p-2)p,$ and $(p-1)p$ modulo $p^2$ since numbers in these classes are divisible by $p$ but not by $p^2$.

Given $p_1, \ldots, p_N$ distinct primes in $P$, the Chinese Remainder Theorem tells us that elements of the set $A$ occupy $\prod_{i=1}^N \big(p_i^2 - (p_i-1) \big)$ modulo $\prod_{i=1}^N p_i^2$ residue classes. Therefore,
\[\upden(A) \defeq \limsup_{N \to \infty} \frac{\big| A \cap \{1, \ldots, N \}\big|}{N} \leq \prod_{i=1}^N \frac{p_i^2 - (p_i-1)}{p_i^2} \leq \prod_{i=1}^N \left(1 - \frac 1 {2p_i} \right).\]

Now we will use the fact that $P$ is divergent to show that the right hand side of the previous expression can be made arbitrarily small. Let $\eps > 0$ and choose $p_1, \ldots, p_N$ distinct primes in $P$ so that $\sum_{i=1}^N p_i^{-1} > 2 \log (\eps^{-1})$. Using that $\log(1-x) < -x$ for $0 < x < 1$,
\[\log \prod_{i=1}^N \left(1 - \frac 1 {2p_i} \right) = \sum_{i=1}^N \log \left(1 - \frac 1 {2p_i} \right) \leq  - \frac 12 \sum_{i=1}^N p_i^{-1} < \log \eps.\]
This shows that $\upden(A) = 0$, whereby $\den(A) = 0$.
\end{proof}

To complete the proof that $\den(\NPsi) = 0$, we have only to show that the set of primes which are congruent to 3 modulo 4 is divergent. This was shown first by Dirichlet in 1837: primes equidistribute (with respect to the \emph{Dirichlet density}) amongst residue classes coprime to the modulus. Strengthening and generalizing this result, the Prime Number Theorem (PNT) of Hadamard and de la Vall\'ee-Poussin from 1896 gives asymptotics for the number of primes in arithmetic progressions.\footnote{de la Vall\'ee Poussin went on to give asymptotics for the set of primes represented by a given binary quadratic form; see \cite[Chapter 5]{narkiewiczbook}.} We use the notation $f(N) \sim g(N)$, $N \to \infty$, to mean that $f(N)\big/ g(N) \to \infty$ as $N \to \infty$, and we denote by $\varphi$ Euler's totient function.

\begin{pnt}[{\cite[Theorem 5.14]{narkiewiczbook}}]
If $a, b \in \N$ are coprime, then
\[\big| \P \cap (a \N + b) \cap \{1, \ldots, N\} \big| \sim \frac 1{\varphi(a)} \frac{N}{\log N}, \qquad N \to \infty;\]
in particular, $\denrel \big( \P \cap (a \N + b) \big) = 1 /\varphi(a)$.
\end{pnt}

Now we turn to showing that $\Psi$ represents a positive relative density of the primes numbers. The following classic fact was proven first by Euler around 1750 but was known already to Fermat a hundred years earlier.

\begin{fermat}[{\cite[Theorem 1.2]{coxbook}}]\label{thm:psumoftwosquares}
An odd prime $p \in \P$ is representable as a sum of two squares if and only if it is congruent to 1 modulo 4.
\end{fermat}

\noindent It follows from Fermat's theorem that $\PPsi = \{2\} \cup \big(\P \cap (4 \N + 1) \big)$. By the PNT,
\[\denrel \big( \PPsi \big) = \denrel \big( \{2\} \cup \big(\P \cap (4 \N + 1) \big) \big) = \frac 12,\]
completing the proof of the main theorem for the norm form $x^2 + y^2$.

We will follow this outline to prove the main theorem in its full generality. Thus, to show that $\den(\NPsi) = 0$, we will prove an analogue of Lemma \ref{lem:ifpdividesthensodoesptwo} and then apply Lemma \ref{lem:zerodensityfromCRT}. To show that $\denrel(\PPsi) > 0$, we will give a description of $\PPsi$ and apply a generalization of the PNT. In both parts, we require an understanding of how primes behave in the number field from which $\Psi$ arises, so we proceed by reviewing the necessary notation and terminology. This material is standard and can be found, for example, in \cite[Section 5]{coxbook} or \cite[Chapter 3]{marcusbook}.

\section{Prime ideals and example revisited}\label{sec:primeideals} 

Let $K$ be a number field. Every non-zero, proper ideal of the ring of integers $\intringk$ of $K$ factors uniquely into a product of prime ideals (that is, $\intringk$ is a \emph{Dedekind domain}). The \emph{(absolute) norm} of an ideal $I \subseteq \intringk$ is $\idealnorm(I) = |\intringk \big/ I|$. The norm of the product of ideals is the product of their norms, and the norm of a principal ideal is $\idealnorm(\omega \intringk) = |\normkq(\omega)|$. Every non-zero prime ideal $\pprime$ of $\intringk$ is maximal, hence the \emph{residue field} $\intringk / \pprime$ is a finite field and the norm of $\pprime$ is a prime power.  We say that a prime ideal $\pprime$ is \emph{degree 1} if its norm $\idealnorm(\pprime)$ is prime.

Let $L$ be a finite extension of $K$, and denote its ring of integers by $\intringl$. Since we will frequently need to refer to prime ideals of $\intringk$ and $\intringl$, it will be convenient to call such ideals simply \emph{primes of $K$ and $L$}.

Let $\bprime$ be a prime of $L$. The ideal $\pprime = \bprime \cap \intringk$ is the prime of $K$ that \emph{lies below} $\bprime$ (we also say $\bprime$ \emph{lies above} $\pprime$). This leads to the \emph{residue field extension} $\intringl \big/ \bprime \supseteq \intringk \big/ \pprime \supseteq \Z \big/ p\Z$, where $p\Z$ is the prime of $\Q$ that lies below $\bprime$ (and below $\pprime$). Note that if $\bprime$ is degree 1, then $\intringl \big/ \bprime = \intringk \big/ \pprime = \Z \big/ p\Z$.

Going the other direction, let $\pprime$ be a prime of $K$. The ideal $\pprime \intringl$ factors (\emph{splits}) uniquely into a product of those primes $\bprime_i$ of $L$ which lie above $\pprime$:
\begin{align}\label{eqn:primesplit}\pprime \intringl = \bprime_1^{e_1} \cdots \bprime_n^{e_n}.\end{align}
We call $e_i$ and $f_i \defeq \big[ \intringl \big/ \bprime_i : \intringk \big/ \pprime \big]$ the \emph{ramification index} and \emph{inertia degree of $\bprime_i$ over $\pprime$}, respectively. If all $e_i = 1$, then $\pprime$ is \emph{unramified in $L$}, and otherwise it \emph{ramifies in $L$}. If all $e_i = 1$ and all $f_i = 1$, then $p$ is said to \emph{split completely in $L$}. Taking the norm of both sides of (\ref{eqn:primesplit}), we see that $[L:K] = \sum_{i=1}^n e_if_i$.

When $\loverk$ is Galois, the Galois group $G = \galloverk$ acts transitively on the set of primes $\{\bprime_i\}$ of $L$ above a given prime $\pprime$ of $K$. It follows in this case that all ramification indices $e_i$ are equal and all inertia degrees $f_i$ are equal, meaning $[L:K] = n e f$.\\

The main theorem translates nicely into a result about the set of norms of the principal ideals of $\intringk$. Indeed, if the norm form $\Psi$ arises from $K$, then
\begin{align*}\label{eqn:newcharacterizationofnpsi}\NPsi &= \big\{ |\normkq(\omega)| \ \big| \ \omega \in \intringk \setminus \{0\} \big\} \\
\notag &= \big\{ \idealnorm(I) \ \big| \ \text{$I$ a  principal ideal of $\intringk$} \big\}, \\
\PPsi &= \big\{ \idealnorm(\pprime) \ \big| \ \text{$\pprime$ a  principal, degree 1 prime of $K$} \big\}.
\end{align*}
This description of $\NPsi$ and $\PPsi$ shows explicitly that these sets depend not on the specific form of $\Psi$ but rather only on the number field $K$ itself.

Reformulating the results for $x^2 + y^2$ from the previous section in terms of ideals will indicate how to proceed with the general case. The norm form $\Psi(x,y)=x^2+y^2$ arises from $\gaussfield$ with the basis $(1,i)$. The ring of integers of $\gaussfield$, the Gaussian integers $\gaussint$, form a Euclidean domain, so all ideals are principal and there is unique factorization of elements into primes.

The prime factorization of $2\in \P$ in $\gaussint$ is $2 = (1+i)(1-i)$. Since $(1+i) \gaussint = (1-i) \gaussint$, we see $2\gaussint = \big((1+i) \gaussint \big)^2$, meaning $2\Z$ ramifies in $\gaussfield$. In general, only finitely many primes ramify in any given extension of number fields; see \cite[Theorems 24 and 34]{marcusbook}. Since the density of a set is unchanged with the inclusion or exclusion of finitely many elements, we will safely be able to restrict our attention to unramified primes in the next sections.

If $p \in \P$ is congruent to 3 modulo 4, then $p$ remains prime in $\gaussint$. Thus, $p\gaussint$ is a prime of $\gaussfield$ with norm $\idealnorm(p \gaussint) = p^2$. Of particular importance is that there is no degree 1 prime of $\gaussfield$ above $p\Z$; this fact, combined with the following lemma, yields another proof of Lemma \ref{lem:ifpdividesthensodoesptwo}.

\begin{lemma}\label{lem:highprimes}
Let $\Psi$ be a norm form arising from the number field $K$, and let $n \in \NPsi$. Suppose $p \in \P$ is such that there is no degree 1 prime of $K$ above $p\Z$. If $p|n$, then $p^2 | n$.
\end{lemma}

\begin{proof}
Since $n \in \NPsi$, there exists a (principal) ideal $I \subseteq \intringk$ with $\idealnorm(I) = n$. Suppose $p | n$. Factoring $I$ into a product of primes of $K$, we see that there exists a prime $\pprime$ of $K$ with $p \big| \idealnorm(\pprime)$. This means $\pprime$ lies above $p\Z$. By assumption, $\pprime$ is not degree 1, whereby $p^2 | \idealnorm(\pprime)$. It follows that $p^2 \big | \idealnorm(I)$, meaning $p^2 | n$.
\end{proof}

If $p \in \P$ is congruent to 1 modulo 4, then $p = x^2 + y^2 = (x+iy)(x-iy)$ is the prime factorization of $p$ in $\gaussint$. It follows that $p\gaussint = (x+iy) \gaussint (x-iy) \gaussint$, meaning $p\Z$ splits completely in $\gaussfield$. Of particular importance to us is that there exists a principal, degree 1 prime of $\gaussfield$ above $p\Z$.

\begin{lemma}\label{lem:splitprimes}
Let $\Psi$ be a norm form arising from the number field $K$. The prime $p\in \P$ is represented by $\Psi$ ($p \in \PPsi$) if and only if at least one the primes of $K$ above $p\Z$ is principal and of degree 1. In other words,
\[\PPsi = \big\{ p \in \P \ \big| \ \text{there is a principal, degree 1 prime of $K$ above $p\Z$} \big\}.\]
\end{lemma}

\begin{proof}
If $p \in \PPsi$, then there is a principal ideal $\pprime$ of $\intringk$ with $\idealnorm(\pprime) = p$. Since $\idealnorm(\pprime) = p$, the ideal $\pprime$ is a degree 1 prime above $p\Z$. If, conversely, $\pprime$ is a principal, degree 1 prime above $p\Z$, then $\idealnorm(\pprime) = p$, whereby $p \in\PPsi$.
\end{proof}

Lemmas \ref{lem:highprimes} and \ref{lem:splitprimes} tell us how to proceed with proving the main theorem for the general norm form $\Psi$. To show that $\den(\NPsi) = 0$, we must find a divergent set of rational primes with the property that there are no degree 1 primes of $K$ above them. To show that $\denrel(\PPsi) > 0$, we must show that the set of rational primes above which lie a principal, degree 1 prime of $K$ has positive density. Assuming we have a generalized PNT, there are still two main difficulties to overcome: handling the splitting of primes in non-Galois extensions and understanding when prime ideals are principal. Tools from class field theory help us overcome both of these difficulties.

\section{Norm forms represent few integers}\label{sec:fewintegers} 

Let $\Psi$ be a norm form arising from the number field $K$. In this section we will show that $\den(\NPsi) = 0$. We have reduced this by Lemmas \ref{lem:zerodensityfromCRT} and \ref{lem:highprimes} to showing that
\[\Phigh(K) \defeq \big\{ p \in \P \ \big| \ \text{there is no degree 1 prime of $K$ above $p\Z$} \big\}\]
is a divergent set of primes. We will prove this by describing the set $\Phigh(K)$ with the help of the Artin symbol and then using a generalization of the PNT.

\begin{artinsymbol}[{\cite[Lemma 5.19 \& Corollary 5.21]{coxbook}}]\label{lem:artinprops}
Suppose $\loverk$ is Galois. Let $\bprime$ and $\pprime$ be primes of $L$ and $K$, respectively, with $\pprime$ unramified in $L$ and $\bprime$ lying above $\pprime$. There is a unique element of $\gal(\loverk)$, denoted by the \emph{Artin symbol} $\artin{L}{K}{\bprime}$, with the property that for all $x \in \intringl$,
\[\artin{L}{K}{\bprime}(x) \equiv x^{\idealnorm(\pprime)} \pmod \bprime.\]
Moreover, for all $\sigma \in \gal(\loverk)$, $\artin{L}{K}{\sigma(\bprime)} = \sigma \artin{L}{K}{\bprime} \sigma^{-1}$.
\end{artinsymbol}

Suppose that $\loverq$ is Galois. Since $G = \gal(\loverq)$ acts transitively on the set of primes of $L$ over a given prime of $\Q$, the Artin symbol gives a natural way to associate a conjugacy class of $G$ to a rational prime that is unramified in $L$. More precisely, given a conjugacy class $\cclass \sigma \defeq \{ g \sigma g^{-1} \ | \ g \in G \}$ of the Galois group $G$, the set of rational primes corresponding to it is
\[\Ploverq \big(\cclass \sigma \big) \defeq \left\{ p \in \P \ \middle| \ \begin{aligned} & \text{$p\Z$ is unramified in $L$, and} \\ & \text{there is a prime $\bprime$ of $L$ with $\artin{L}{\Q}{\bprime} = \sigma$ above $p\Z$} \end{aligned} \right\}.\]
The Artin-Chebotarev density theorem is a generalized prime number theorem that gives asymptotics for these sets of primes.
\begin{cdensitytheorem}[{\cite[Theorem 3.4]{serrebooktwo}}\footnote{The Chebotarev density theorem as originally formulated by Chebotarev \cite[Hauptsatz]{chebotarevpaper} generalizes Dirichlet's theorem on arithmetic progressions. The formulation here is due to Artin \cite[Satz 4]{artinpaper} and generalizes the PNT. See \cite[Chapter 3]{serrebooktwo} for more discussion.}]
Let $\loverq$ be Galois. For all conjugacy classes $\cclass \sigma \subseteq \gal(\loverq)$,
\[\big| \Ploverq (\cclass \sigma) \cap \{1, \ldots, N\} \big| \sim \frac {\big|\cclass \sigma \big|}{\big|\gal(\loverq)\big|} \frac{N}{\log N}, \qquad N \to \infty;\]
in particular, $\denrel \big( \Ploverq (\cclass \sigma) \big) = \big|\cclass \sigma \big| \big / \big|\gal(\loverq)\big|$.
\end{cdensitytheorem}

Since the extension $\koverq$ we care about may not be Galois, our first step going forward is to let $L$ be the \emph{Galois closure} of $\koverq$ in $\mathbb{C}$: it is the minimal Galois extension of $\Q$ in $\mathbb{C}$ containing $K$. Let $G = \gal(\loverq)$ and $H = \gal(\loverk)$.

The next step is to describe $\Phigh(K)$ in terms of the sets $\Ploverq \big(\cclass \sigma\big)$. To do that, we need a connection between degree 1 primes of $K$ and the Artin symbol. The following lemma gives such a connection and is key to both of the conclusions in the main theorem.

\begin{lemma}[{\cite[Page 133]{neukirchbook}}]\label{lem:descriptionofhighprimes}
The set $\Phigh(K)$ consists of (up to finitely many exceptions) those rational primes corresponding to conjugacy classes of $G$ that do not intersect $H$:
\begin{align}\label{eqn:descriptofphigh}\Phigh(K) \ \dot= \ \bigcupdot_{\substack{\cclass \sigma \subseteq G \\ \cclass \sigma \cap H = \emptyset}} \Ploverq (\cclass \sigma),\end{align}
where $\dot=$ indicates that the set equality is up to finitely many elements and $\cupdot$ denotes a disjoint union.
\end{lemma}

\begin{proof}
We will show the complementary statement, namely that
\begin{align}\label{eqn:defofplow}\Plow(K) & \defeq \big\{ p \in \P \ \big| \ \text{there exists a degree 1 prime of $K$ above $p\Z$} \big\}\\
\notag & \; \; \dot= \ \bigcupdot_{\substack{\cclass \sigma \subseteq G \\ \cclass \sigma \cap H \neq \emptyset}} \Ploverq (\cclass \sigma).
\end{align}
We will only consider rational primes that are unramified in $L$, so the descriptions of $\Plow(K)$ and $\Phigh(K)$ will be up to a finite set of exceptions. The fact that the union above is disjoint follows immediately from the definition of $\Ploverq (\cclass \sigma)$.

To show that $\Plow(K) \ \dot\subseteq \ \bigcupdot \Ploverq (\cclass \sigma)$, let $p \in \Plow(K)$, and suppose $\bprime$ is a prime of $L$ so that $\pprime = \bprime \cap \intringk$ is a degree 1 prime of $K$.  Let $\sigma = \artin{L}{\Q}{\bprime}$. Since $p \in \Ploverq (\cclass \sigma)$, we only need to show that $\cclass \sigma \cap H \neq \emptyset$.

Since $\bprime$ lies above $p\Z$, by the definition of the Artin symbol, for all $x \in \intringl$,
\[\sigma(x) \equiv x^p \pmod \bprime.\]
Since $\pprime$ is degree 1, $\idealnorm(\pprime) = p$, whence for all $x \in \intringl$,
\[\sigma(x) \equiv x^p = x^{\idealnorm(\pprime)} \equiv \artin LK{\bprime}(x) \pmod \bprime,\]
again by the definition of the Artin symbol.
It follows now by the uniqueness of the Artin symbol that $\sigma = \artin LK{\bprime}$. Since $\artin LK{\bprime} \in H$, $\cclass \sigma \cap H \neq \emptyset$.

To show that $\Plow(K) \ \dot \supseteq \ \bigcupdot \Ploverq (\cclass \sigma)$, let $p \in \Ploverq (\cclass \sigma)$ for some $\cclass \sigma \cap H \neq \emptyset$; without loss of generality, we may assume $\sigma \in H$. Let $\bprime$ be a prime of $L$ above $p\Z$ with $\artin L{\Q}{\bprime} = \sigma$, and let $\pprime = \bprime \cap \intringk$. For all $x \in \intringl$,
\[\sigma(x) \equiv x^p \pmod \bprime.\]
Since $H = \galloverk$, $\sigma$ fixes $\intringk$, thus for all $x \in \intringk$,
\begin{align*}\label{eqn:xtothepisx}\sigma(x) = x \equiv x^p \pmod \bprime.\end{align*}
It follows that for all $x \in \intringk \big / \pprime \subseteq \intringl \big / \bprime$, $x = x^p$. There are exactly $p$ solutions to the equation $x = x^p$ in the finite field $\intringl \big / \bprime$ (satisfied by elements of the base field), hence $\idealnorm(\pprime) = \big| \intringk \big / \pprime \big| = p$. This means that $\pprime$ is a degree 1 prime of $K$ over $p\Z$.
\end{proof}

To show that $\Phigh(K)$ is a divergent set of primes, we must show that the union in (\ref{eqn:descriptofphigh}) is non-empty.  The fact that there are conjugacy classes of the Galois group that avoid $H$ follows from the following general fact about finite groups.

\begin{lemma}\label{lem:finitegroupfact}
Let $G$ be a finite group and $H$ be a proper subgroup. There exists $g \in G$ for which $\cclass g \cap H = \emptyset$.
\end{lemma}

\begin{proof}
The group $G$ acts on its subgroups by conjugation.  By the Orbit-Stabilizer theorem, the orbit $\{H_i\}_{i=1}^n$ of $H$ has cardinality $n = |G| \big/ |N_G(H)|$, where $N_G(H)$ is the normalizer subgroup of $H$. Since $H \subseteq N_G(H)$, $n \leq |G| \big/ |H|$. Note that all $|H_i|$ are equal to $|H|$ and for $i \neq j$, $|H_i \cap H_j| \geq 1$. It follows that
\[\left| \bigcup_{x \in G} x H x^{-1} \right| = \left|\bigcup_{i=1}^n H_i \right| < n |H| \leq |G|,\]
whereby there exists $g \in G$ with the property that for all $x \in G$, $g \notin x H x^{-1}$. This means that for all $x \in G$, $x^{-1} g x \not \in H$, meaning $\cclass g \cap H = \emptyset$.
\end{proof}

We can now prove part (I) of the main theorem in full generality.

\begin{proof}[Proof of Main Theorem (I)]
By the Artin-Chebotarev Density Theorem, for all conjugacy classes $\cclass \sigma \subseteq \gal(\loverq)$, the quantity $\denrel \big(\Ploverq (\cclass \sigma) \big)$ exists and is positive. By Lemma \ref{lem:descriptionofhighprimes} and the fact that the union in (\ref{eqn:descriptofphigh}) is a disjoint union,
\[\denrel \big(\Phigh(K) \big) = \sum_{\substack{\cclass \sigma \subseteq G \\ \cclass \sigma \cap H = \emptyset}} \denrel \big(\Ploverq (\cclass \sigma) \big).\]
By Lemma \ref{lem:finitegroupfact}, the sum in the previous expression is non-empty; in particular, $\denrel \big(\Phigh(K) \big)$ exists and is positive. It follows that $\Phigh(K)$ is a divergent set of primes. The proof is now completed by Lemmas \ref{lem:zerodensityfromCRT} and \ref{lem:highprimes}.
\end{proof}

\section{Norm forms represent relatively many primes}\label{sec:manyprimes} 

Let $\Psi$ be a norm form arising from the number field $K$. In this section we will show that $\denrel(\PPsi) > 0$. By Lemma \ref{lem:splitprimes},
\[\PPsi = \big\{ p \in \P \ \big| \ \text{there is a principal, degree 1 prime of $K$ above $p\Z$} \big\},\]
thus we must show that the relative density of this set exists and is positive. We will accomplish this by proving that $\PPsi \ \dot= \ \Plow(T)$, where $T$ is some well-chosen extension of $K$.

Lemma 4.1 already provides us with a nice connection between the property of being degree 1 and the Artin symbol. To connect the property of being principal with splitting behavior, we use the Hilbert class field.

\begin{hclassfield}[{\cite[Corollary 5.25 and Theorem 8.19]{coxbook}}]
Let $K$ be a number field. There exists a unique finite Galois extension $T$ of $K$ with the property that for all primes $\pprime$ of $K$, $\pprime$ is principal if and only if $\pprime$ splits completely in $T$.
\end{hclassfield}

We are now in the position to prove part (II) of the main theorem.

\begin{proof}[Proof of Main Theorem (II)]

Let $T$ be the Hilbert class field of $K$. It suffices by Lemma \ref{lem:descriptionofhighprimes} and the \acd{} to prove that $\PPsi \ \dot= \ \Plow(T)$, where $\Plow(T)$ is defined in (\ref{eqn:defofplow}). Using the description of $\PPsi$ in Lemma \ref{lem:splitprimes}, this amounts to showing that there is a principal, degree 1 prime of $K$ above $p\Z$ if and only if there is a degree 1 prime of $T$ above $p\Z$.

Suppose $\pprime$ is a principal, degree 1 prime of $K$ above $p\Z$. Since $T$ is the Hilbert class field of $K$, $\pprime$ splits completely in $T$. Since $\pprime$ is degree 1, any prime of $T$ above $\pprime$ is a degree 1 prime of $T$ above $p\Z$.

Conversely, suppose that $p\Z$ is unramified in $T$ and that $\bprime$ is a degree 1 prime of $T$ above $p\Z$. Since $\bprime$ is degree 1, so is $\pprime \defeq \bprime \cap \intringk$; in particular, the inertia degree of $\bprime$ over $\pprime$ is 1. Since $T \big/ K$ is Galois and $\pprime$ is unramified in $T$, $\pprime$ splits completely in $T$. Since $T$ is the Hilbert class field of $K$, $\pprime$ is a principal, degree 1 prime of $K$ above $p\Z$.
\end{proof}

\section{Arithmetic progressions in $\NPsi$}

Do there exist arbitrarily long arithmetic progressions, each term of which is of the form $x^2 + y^2$? Because $\den(\NPsi) = 0$, we cannot appeal to Szemer\'edi's theorem, which states that \emph{any set $A \subseteq \N$ with $\upden(A)>0$ is AP-rich}, that is, contains arbitrarily long arithmetic progressions.

We can still reach the desired conclusion, however, using the fact that $\denrel(\PPsi) > 0$ and the following famous theorem of Green and Tao.

\begin{greentao}[{\cite{greentaooriginal}}]
Any subset $P \subseteq \mathbb{P}$ of positive relative upper density $\updenrel(P) > 0$ is AP-rich.
\end{greentao}

Combining part (II) of the main theorem with the Green-Tao theorem, we see that not only is $\NPsi$ AP-rich for any norm form $\Psi$, but $\PPsi$ is as well.

Do we know of a way to prove that $\NPsi$ is AP-rich without simultaneously giving the result for $\PPsi$?  If not, this is indeed curious: the only way we know to show that the set of positive integers represented by $\Psi$ is AP-rich is to show the ostensibly much harder result that the set of \emph{primes} represented by $\Psi$ is AP-rich.\\

\noindent \textbf{Acknowledgements} \ Thanks goes to Jim Cogdell, Warren Sinnott, Joel Moreira, and James Maynard for helpful discussions. Gratitude is also extended to Timothy Browning for calling the author's attention to the work of Odoni.

\bibliographystyle{alphanum}
\bibliography{formsbib}

\end{document}